\documentclass[reqno,12pt]{amsart}%
\usepackage{amsfonts}
\usepackage{amsmath}
\usepackage{amssymb}

\def\bE{{\mathbb{E}}}

\def\bR{{\mathbb{R}}}
\def\bRd{{\mathbb{R}^{\mathrm{d}}}}

\def\cE{{\mathcal{E}}}

\def\cJ{{\mathcal{J}}}
\def\cC{{\mathcal{C}}}
\def\cL{{\mathcal{L}}}

\def\ba{\boldsymbol{\alpha}}

\newcommand\zm{(\boldsymbol{0})}
\newcommand\km{\boldsymbol{\epsilon}(k)}

\def\mfm{\mathfrak{m}}
\def\mfq{\mathfrak{q}}
\def\Hep{{\mathrm{H}}}

\newcommand\mfk[1]{\mathfrak{#1}}

\def\ba{\boldsymbol{\alpha}}

\newtheorem{theorem}{Theorem}
\theoremstyle{plain}
\newtheorem{corollary}[theorem]{Corollary}
\newtheorem{definition}[theorem]{Definition}
\newtheorem{proposition}[theorem]{Proposition}

\newtheorem{remark}[theorem]{Remark}

\newcommand\bel[1]{\begin{equation}\label{#1}}
\newcommand\ee{\end{equation}}

\numberwithin{theorem}{section}
\numberwithin{equation}{section}

\begin{document}
\title[Fractional SDEs]
{Classical and Generalized Solutions of Fractional Stochastic Differential
Equations}

\author{S. V. Lototsky}
\curraddr[S. V. Lototsky]{Department of Mathematics, USC\\
Los Angeles, CA 90089}
\email[S. V. Lototsky]{lototsky@math.usc.edu}
\urladdr{http://www-bcf.usc.edu/$\sim$lototsky}

\author{B. L. Rozovsky}\thanks{Research supported by ARO Grant W911N-16-1-0103}
\curraddr[B. L. Rozovsky]{Division of Applied Mathematics\\
Brown University\\
Providence, RI 02912}
\email[B. L. Rozovsky]{Boris\_Rozovsky@Brown.edu}
\urladdr{www.dam.brown.edu/people/rozovsky/rozovsky.htm}



\begin{abstract}
For stochastic evolution equations with fractional derivatives,
classical solutions exist when the order of the time derivative of the
unknown function  is not too small compared to the order of the time derivative
of the noise; otherwise, there can be  a generalized solution in suitable
weighted chaos spaces. Presence of fractional  derivatives in time leads to
various modifications of the stochastic parabolicity condition. 
Interesting new effects appear when
the order of the time derivative in  the noise term is less than or equal to
 one-half.
\end{abstract}

\maketitle

\today

\section{Introduction}
Given a $\beta\in (0,1)$, and a smooth function
 $f=f(t),\  t>0$, the two most popular definitions of
 the derivative of order $\beta$   are
 Riemann-Liouville
 $$
 D^{\beta}_tf(t)=\frac{1}{\Gamma(1-\beta)}
\frac{d}{dt}\int_0^t  (t-s)^{-\beta} f(s) \, ds\
$$
and Caputo
$$
  \tilde{\partial}^{\beta}_tf(t)=\frac{1}{\Gamma(1-\beta)}
\int_0^t  (t-s)^{-\beta} f'(s)\big) \, ds;
 $$
 \bel{Gamma}
 \Gamma(z)=\int_0^{+\infty} t^{z-1}e^{-t}\, dt.
 \ee

 The Riemann-Liouville derivative can be considered a
  true extension of the usual derivative to fractional orders.
  For example, a function does not have to be continuously
 differentiable to have  Riemann-Liouville derivatives of order $\beta<1$
 \cite{FD-counterex}. On the other hand,
 the Caputo derivative is more convenient in
 initial-value problems, with no need for fractional-order initial
 conditions \cite[Section 2.4.1]{Podlubny}.

 The Kochubei extension of the  Caputo derivative,
$$
\partial^{\beta}_tf(t)=\frac{1}{\Gamma(1-\beta)}
\frac{d}{dt}\int_0^t  (t-s)^{-\beta} \big(f(s)-f(0+)\big) \, ds,
$$
$f(0+)=\lim_{t\to 0,t>0} f(t)$, seems to achieve the right balance
 between  mathematical utility and physical relevance  \cite{Kochubei}
 and has been recently used in the study of large classes of stochastic partial
 differential  equations \cite{Kim-Fract, LRdS}.

Let $w=w(t), \ t\geq 0,$ be a standard Brownian motion
on a stochastic basis $(\Omega,\mathcal{F},
\{\mathcal{F}_t\}_{t\geq 0},\mathbb{P})$ satisfying the usual
conditions. The objective of this paper is to address fundamental
questions about existence and regularity of solution for equations of the type
\bel{eq:main0000}
\partial^{\beta}_t X(t) = aX(t) +\partial^{\gamma}_t\int_0^t \big(\sigma X(s)
+g(s)\big)\, dw(s),\ t>0,\ \ a,b\in \bR.
\ee
With a suitable choice of $a$ and $\sigma$, \eqref{eq:main0000} covers the time-fractional versions of the Ornstein-Uhlenbeck process and the geometric Brownian motion, as well as certain evolution equations in function spaces.
The emphasis is on derivation and analysis of explicit formulas for the solution,
 as opposed to the development of general theory.

Given the vast literature on the subject of fractional differential equations,
Section  \ref{sec:Bg} provides the necessary background, to make
the presentation reasonably self-contained. Section \ref{sec:FDBM}
investigates the  equation with $a=\sigma=0$, corresponding to fractional
derivatives or integrals of the Brownian motion, followed
by the time fractional Ornstein-Uhlenbeck process ($\sigma=0$)
 in Section \ref{sec:TFOU} and the geometric Brownian motion
 in Section \ref{sec:TFGBM}. Along the way we understand  the
 origins of the condition $\beta-\gamma>-1/2$ at a more basic
 level than in \cite{Kim-Fract, LRdS} and discover that the
 fractional in time Ornstein-Uhlenbeck process can exhibit the
 full range of sub-diffusive behaviors, including  super-slow
 logarithmic.  Section \ref{sec:SFPC} investigates an
 SPDE version of \eqref{eq:main0000} by replacing the numbers
 $a,\sigma$ with fractional powers of the Laplace operator. Then the
 results of the previous sections lead to
 several  versions of the  stochastic parabolicity condition.

Throughout the paper,
 $\cC(G)$ denotes the space of real-valued continuous functions on $G$
 and  $\cC_{loc}(G)$ is the space of functions that are
  continuous on every compact sub-set of $G$;  $\Gamma=\Gamma(z)$,
  is    the Gamma function, defined for
all complex $z$ except for the poles at $0,-1,-2,\ldots$ and,
for $z$ in the right half-plane, having the representation \eqref{Gamma}.
Most of other notations, such as $\cL[\cdot]$ and $\cE$ for the Laplace transform and its domain, and $E_{\beta,\rho}$ for the two-parameter
 Mittag-Leffler function, are introduced in Section \ref{sec:Bg}.

\section{Background}
\label{sec:Bg}
 \subsection{Fractional Derivatives and Integrals}

  In this section we do not indicate the time
  variable as a subscript in the notations of the derivatives: for $\beta\in (0,1),$
  \begin{align}
  \label{RL-d1}
  D^{\beta}f(t)&=\frac{1}{\Gamma(1-\beta)}
\frac{d}{dt}\int_0^t  (t-s)^{-\beta} f(s) \, ds,\\
\label{C-d1}
  \partial^{\beta}f(t)&=\frac{1}{\Gamma(1-\beta)}
\frac{d}{dt}\int_0^t  (t-s)^{-\beta}\big( f(s)-f(0)\big) \, ds.
  \end{align}
  We also introduce the corresponding  fractional integrals: for $p>0$,
  \begin{align}
  \label{RL-int}
  I^{p}f(t)&=\frac{1}{\Gamma(p)}
 \int_0^t  (t-s)^{p-1} f(s) \, ds,\\
 \label{C-int}
  J^{p}f(t)&=\frac{1}{\Gamma(p)}
 \int_0^t  (t-s)^{p-1} \big(f(s)-f(0+)\big)\, ds,
  \end{align}
where  $f(0+)=\lim_{t\to 0} f(t)$.
 In what follows, with all functions defined only for $t>0$, we
 write $f(0)$  instead of $f(0+)$
By convention, $I^0f(t)=f(t),\ J^0f=f(t)-f(0)$.
In particular, for the constant function $f(t)=1$, $t\geq 0$,
 \bel{der-const}
 I^{p}[1](t)=\frac{t^p}{\Gamma(p+1)},\ \ \ J^p[1](t)=0,\  \ \ D^{\beta}[1](t)=\frac{t^{-\beta}}{\Gamma(1-\beta)},\ \ \
 \partial^{\beta}[1](t)=0.
 \ee

Formulas \eqref{RL-d1}--\eqref{C-int} imply
 \begin{align}
 \label{RL-DI}
 D^{\beta}f(t)&=\frac{d}{dt} I^{1-\beta}f(t),\\
 \label{C-DJ}
 \partial^{\beta}f(t)&=\frac{d}{dt} J^{1-\beta}f(t).
 \end{align}

 \begin{proposition}
 For all $p,q>0$,
 \bel{IJ-sg}
 I^p(I^qf)=I^{p+q}f,\ J^p(J^qf)=J^{p+q}f.
 \ee
 \end{proposition}

 \begin{proof}
 For $I$,
 \begin{equation*}
 \begin{split}
\Gamma(p)\Gamma(q) I^p(I^qf)(t)&=
 \int_0^t\int_0^s (t-s)^{p-1}(s-r)^{q-1}f(r)\, dr\,ds\\
 &=
 \int_0^t\left(\int_r^t (t-s)^{p-1}(s-r)^{q-1}\, ds \right)f(r)\, dr\\
 &=\mathrm{B}(p,q)\int_0^t (t-r)^{p+q-1}f(r)dr=
 \Gamma(p)\Gamma(q)I^{p+q}f(t).
 \end{split}
\end{equation*}
The same computation works for $J$ after noticing that if $p>0$ and
$f(t)$ is bounded near $0$, then
\bel{lz}
\lim_{t\to 0+} |J^{p}f(t)|\leq C\lim_{t\to 0+} t^p=0.
\ee
\end{proof}

\begin{proposition}
If $f\in \cC_{loc}\big([0,+\infty)\big)$, then
\bel{DI=Id}
\partial^{\beta}I^{\beta}f(t)=D^{\beta}I^{\beta}f(t)=f(t),\ \ \ \ \
\partial^{\beta}J^{\beta}f(t)=f(t)-f(0).
\ee

\end{proposition}

\begin{proof} Using
 \eqref{RL-DI}, \eqref{C-DJ}, \eqref{IJ-sg} and keeping in mind
that, similar to \eqref{lz},
 $I^{\beta}f(0)=0$, the result follows from  the fundamental theorem of
calculus:
\begin{align*}
\partial^{\beta}I^{\beta}f(t)=D^{\beta}I^{\beta}f(t)=
\frac{d}{dt}\big(I^{1-\beta}I^{\beta}f\big)(t)
\frac{d}{dt}\big(I f\big)(t)=\frac{d}{dt}\int_0^tf(s)ds=f(t).
\end{align*}
Similarly,
\begin{align*}
\partial^{\beta}J^{\beta}f(t)=
\frac{d}{dt}\big(J^{1-\beta}J^{\beta}f\big)(t)=
\frac{d}{dt}\big(Jf\big)(t)=\frac{d}{dt}\int_0^t\big(f(s)-f(0)\big)ds
=f(t)-f(0).
\end{align*}
\end{proof}

\begin{proposition}
If
\bel{AC}
f(t)=f(0)+\int_0^t f'(s)\,ds,
\ee
then
\begin{align}
\label{IJ}
J^pf=I^{1+p}f',\\
\label{C-d2}
\partial^{\beta}f=I^{1-\beta}f'.
\end{align}
\end{proposition}

\begin{proof}
For \eqref{IJ}, integrate by parts:
\begin{align*}
p\Gamma(p)J^pf(t)&=\int_0^t \left(-\frac{\partial}{\partial s}
(t-s)^p\right)\left(\int_0^s f'(r)\,dr\right)\, ds\\
&=(t-s)^p\left(\int_0^s f'(r)\,dr\right)\Bigg|_{s=0}^{s=t}\!\!+
\int_0^t (t-s)^pf'(s)\,ds=\Gamma(1+p)I^{1+p}f'(t),
\end{align*}
and remember that $p\Gamma(p)=\Gamma(1+p)$.

For \eqref{C-d2}, differentiate  \eqref{IJ} taking $p=1-\beta$.
\end{proof}

\begin{corollary}
If \eqref{AC} holds, then
\bel{IC=Id}
I^{\beta}\partial^{\beta}f(t)=f(t)-f(0).
\ee
\end{corollary}

\begin{proof}
By \eqref{IJ-sg} and \eqref{C-d2},
$$
I^{\beta}\partial^{\beta}f(t) = I^{\beta}I^{1-\beta}f'(t)=
If'(t)=\int_0^t f'(s)ds=f(t)-f(0).
$$
\end{proof}

\subsection{The Laplace Transform}

Recall that
\bel{LaplTr}
f=f(t) \mapsto \cL[f](\lambda)=\int_0^{+\infty} f(t)e^{-\lambda t}\, dt
\ee
is  a one-to-one mapping defined on
\bel{ExpCl}
\mathcal{E}=\Big\{f\in L_{1,loc}((0,+\infty)):
\sup_{t>0} |f(t)|e^{ct}<\infty\ \ {\rm for \ \ some\ \ }
c\in \bR\Big\}.
\ee
We will use the following properties of the Laplace transform:
\begin{align}
\label{LT-int}
&\cL[If](\lambda)=\lambda^{-1}\cL[f](\lambda);\\
\label{LT-der}
&f'\in \cE\ \Rightarrow\  \cL[f'](\lambda)=\lambda\cL[f]-f(0);\\
\label{LT-conv}
&h(t)=\int_0^t f(t-s)g(s)\, ds \ \Rightarrow \ \cL[h](\lambda)=
\cL[f](\lambda)\,\cL[g](\lambda);\\
\label{LT-t-gm}
&f(t)=\frac{t^{\gamma-1}}{\Gamma(\gamma)}, \ \gamma>0 \
\Rightarrow \ \cL[f](\lambda)=\lambda^{\gamma}.
\end{align}

We now establish  fractional versions of \eqref{LT-int} and \eqref{LT-der}.
\begin{proposition}
\label{prop:LT}
If $f\in\cE\cap\mathcal{C}_{loc}\big([0,+\infty)\big)$
 and $\beta\in (0,1)$, then
\begin{align}
\label{LT-I}
\cL[I^pf](\lambda)&=\lambda^{-p}\cL[f](\lambda),\\
\label{LT-RL}
\cL[D^{\beta} f](\lambda)&=\lambda^{\beta}\cL[f](\lambda),\\
\label{LT-J}
\cL[J^pf](\lambda)&=\lambda^{-p}\cL[f](\lambda)-\lambda^{-p-1}f(0),\\
\label{LT-C}
\cL[\partial^{\beta} f](\lambda)&=\lambda^{\beta}\cL[f](\lambda)-
\lambda^{\beta-1}f(0).
\end{align}
\end{proposition}

\begin{proof}
Equality \eqref{LT-I} is an immediate consequence of \eqref{LT-conv},
and then \eqref{LT-RL} follows from \eqref{LT-der} and \eqref{RL-DI}.
To establish \eqref{LT-J}, we write
\begin{align*}
\Gamma(p)\cL[J^{p}f](\lambda)&=
\int_0^{+\infty}\int_0^t(t-s)^{p-1}\big(f(s)-f(0)\big)
e^{-\lambda t}\,dt\\
&=
\int_0^{+\infty}
\left(\int_s^{+\infty} (t-s)^{p-1}e^{-\lambda t}\, dt\right)
\big(f(s)-f(0)\big) \, ds\\
&=\int_0^{+\infty}
\left(\int_0^{+\infty} u^{p-1}e^{-\lambda u}\, dt\right)
e^{-\lambda s}\big(f(s)-f(0)\big) \, ds\\
&=
\Gamma(p)
\Big(\lambda^{p}\cL[f](\lambda)-\lambda^{p-1}f(0)\Big).
\end{align*}
Then \eqref{C-DJ} and \eqref{LT-der} imply \eqref{LT-C}.
\end{proof}

Next, we compute
 the Laplace transform of the standard Brownian motion. Define
\bel{LT-dw}
\hat{w}(\lambda)=\int_0^{+\infty}e^{-\lambda t}\,dw(t);
\ee
for every $\lambda>0$, the random variable $\hat{w}(\lambda)$
is Gaussian with mean zero and variance $1/(2\lambda)$.
Then the stochastic Fubini theorem shows that
\bel{LT-w}
\cL[w](\lambda)=\frac{\hat{w}(\lambda)}{\lambda}.
\ee

\subsection{The two parameter Mittag-Leffler function}
The function is defined by the power series
\bel{ML2}
E_{\beta,\rho}(z)
=\sum_{k=0}^{\infty} \frac{z^k}{\Gamma(\beta k+\rho)},
\ee
and, in a sense, is the fractional version of the exponential function.
If $\beta>0$, then, with  the convention $1/\Gamma(-n)=0$, $n=0,1,2,\ldots$,
the series on the right-hand side of  \eqref{ML2} converges for all
$z$  and $\rho$. The particular case $\rho=1$ is
\bel{ML}
E_{\beta,1}(z):=E_{\beta}(z)=\sum_{k=0}^{\infty} \frac{z^k}{\Gamma(\beta k+1)}.
\ee
Note that  $E_{0}=1/(1-z),\ |z|\leq1$, $E_{1}(z)=e^z$,
and $E_2(z)=\cosh(\sqrt{z})$.

\begin{proposition}[The fractional Gronwall-Bellman inequality]
\label{prop-FGBI}
If
$$
y(t)\leq A(t)+B\int_0^t (t-s)^{\beta-1}y(s)\, ds,
$$
where
$y(t)\geq 0$,
 $A(t)\geq 0$ is non-decreasing,
 $B>0$,
 $\beta>0$,
then
$$
y(t)\leq A(t)E_{\beta}\big( B\Gamma(\beta)t^{\beta}\big).
$$
\end{proposition}
\begin{proof}
See \cite[Corollary 2]{GG}.
\end{proof}

In general,  $E_{\beta,\rho}$ cannot be expressed using elementary
functions, but, for many purposes, the following results suffice.

\begin{proposition}
Let  $\beta\in (0,1)$ and $\rho>0$.
\begin{enumerate}
\item There exist  numbers $C_1,C_2>0$ such that, for
all $t>0$,
\bel{ML-asympt0}
|E_{\beta,\rho}(t)|\leq C_1(1+t)^{(1-\rho)/\beta}e^{t^{1/\beta}}+
\frac{C_2}{1+t}.
\ee
\item There exists a number $C$ so that, for all $t>0$,
\bel{ML-asympt1}
|E_{\beta,p}(-t)|\leq \frac{C}{1+t}.
\ee
\item Moreover, if $\rho\in (0,1)$, then
\bel{ML-asympt1-1}
\lim_{t\to+\infty} tE_{\beta,\rho}(-t)=
\frac{1}{\Gamma(\rho-\beta)},\ \beta\not=\rho;
\ee
\bel{ML-asympt1-22}
\lim_{t\to+\infty} t^2E_{\beta,\beta}(-t)=
-\frac{1}{\Gamma(-\beta)}.
\ee
\end{enumerate}
\end{proposition}

\begin{proof}
See \cite[Theorem 1.5]{Podlubny},
 \cite[Theorem 1.6]{Podlubny}, and
\cite[Theorem 1.4]{Podlubny}, respectively.

\end{proof}

Next, define
$$
y_{\beta,\rho}(t)=t^{\rho-1}E_{\beta,\rho}(at^{\beta}).
$$
\begin{proposition}
For every $a\in \bR$,
the family of functions $y_{\beta,\rho}$, $\beta\in (0,1),\ \rho>0$,
has the following properties:
\begin{align}
\label{MTL-LT-gen}
\cL[y_{\beta,\rho}](\lambda)&=
\frac{\lambda^{\beta-\rho}}{\lambda^{\beta}-a};\\
\label{MTL-con-gen}
\frac{1}{\Gamma(\gamma)}
\int_0^t (t-s)^{\gamma-1}y_{\beta,\rho}(s)\, ds&=
y_{\beta,\rho+\gamma}(t),\ \gamma>0;
\\
\label{MTL-diff-gen}
D^{\gamma}y_{\beta,\rho}(t)&=y_{\beta,\rho-\gamma}(t),\
\gamma\in (0,1),\ \rho>\gamma.
\end{align}
\end{proposition}

\begin{proof}
By \eqref{ML-asympt0}, $y_{\beta,\rho}\in \cE$ for
all $\beta\in (0,1),\ \rho>0$ and $a\in \bR$.
Then
\begin{align*}
\cL[y_{\beta,\rho}](\lambda)&=
\sum_{k\geq 0}\int_0^{+\infty}
\frac{a^kt^{k\beta+\rho-1}}{\Gamma(k\beta+\rho )}
\, e^{-\lambda t}\, {dt}=
\lambda^{-\rho}\sum_{k\geq 0} (a\lambda^{-\beta})^k\\
&=\frac{\lambda^{-\rho}}{1-a\lambda^{-\beta}}=
\frac{\lambda^{\beta-\rho}}{\lambda^{\beta}-a},
\end{align*}
proving \eqref{MTL-LT-gen}. Then, with \eqref{LT-t-gm} in mind,
\eqref{MTL-con-gen} and \eqref{MTL-diff-gen}
follow from \eqref{LT-I} and \eqref{LT-RL}, respectively.

\end{proof}

\subsection{Time fractional linear deterministic equations}
Consider the equation
\bel{DetEq-Gen0}
\partial^{\beta} y(t)=ay(t)+f(t),\ t>0,\ y(0)=y_0,
\ee
with $\beta\in (0,1), \ a\in \bR,\ f\in \cE\cap\cC_{loc}\big([0,+\infty)\big).$

\begin{definition}
A function $y\in \cC_{loc}\big([0,+\infty)\big)$ is called a
classical solution of \eqref{DetEq-Gen0} if
\bel{DetEq-def}
J^{1-\beta}y(t)=\int_0^t \big(ay(s)+f(s)\big)\, ds,\ t>0.
\ee
\end{definition}

The following result is the analogue of
 \cite[Example 4.3]{Podlubny}, where the Riemann-Liouville
 derivative is considered.

\begin{theorem}
\label{th:DetEq-Gen}
The unique solution of \eqref{DetEq-Gen0} in $\cE$ is
\bel{DetEq-Gen}
y(t)=y_0E_{\beta}(at^{\beta})+
\int_0^t (t-s)^{\beta-1}E_{\beta,\beta}\big(a(t-s)^{\beta}\big)
f(s)\, ds.
\ee
\end{theorem}

\begin{proof}
Take the Laplace transform on both sides of \eqref{DetEq-def}
and use \eqref{LT-J}:
$$
\lambda^{\beta-1} \cL[y](\lambda)-\lambda^{\beta-2}y_0=
\lambda^{-1}\Big( a\cL[y](\lambda)+\cL[f](\lambda)\Big),
$$
or
$$
\cL[y](\lambda)=\frac{\lambda^{\beta-1}}{\lambda^{\beta}-a}\,y_0+
\frac{\cL[f](\lambda)}{\lambda^{\beta}-a}.
$$
The conclusion of the theorem  now follows from  \eqref{LT-conv},
\eqref{MTL-LT-gen}, and  uniqueness of the Laplace transform on $\cE$.
\end{proof}

\subsection{Chaos Expansion and Generalized Processes}
Below is a summary of the construction of the
weighted chaos spaces; for details, see
\cite{LR_shir, LR_AP, LRS}.

 Introduce the following objects:
\begin{align*}
\{\mfm_k&=\mfm_k(t),\ t\in [0,T]\},\ {\rm an \
orthonormal   \ basis\ in\ } L_2((0,T)),\\
\cJ&=\left\{\ba=(\alpha_k,\, k\geq 1): \alpha_k\in \{0,1,2,\ldots\},\
|\ba|:=\sum_{k}\alpha_k<\infty\right\}, \\
\xi_{\ba}&= \prod_{k}
 \left(
 \frac{\Hep_{\alpha_{k}}(\xi_{k})}{\sqrt{\alpha_{k}!}}\right),\
\Hep_{n}(x) = (-1)^{n} e^{x^{2}/2}\frac{d^{n}}{dx^{n}}%
e^{-x^{2}/2},\
\xi_k=\int_0^{T}\mfk{m}_k(t)\,dw(t).
\end{align*}

If $\eta\in L_2^W(\Omega)$, that is,
a square-integrable functional of $w(t),\ t\in [0,T]$,
then, by  the Cameron-Martin theorem \cite{CM},
$$
\eta=\sum_{\ba\in \cJ} \bE(\eta\xi_{\ba})\, \xi_{\ba},\ \
\bE\eta^2=\sum_{\ba\in \cJ} \Big|\bE(\eta\xi_{\ba})\Big|^2;
$$
see also \cite[Theorem 5.1.12]{LR-SPDE}. For example,
\bel{w-chaos}
w(t)=\sum_{k\geq 1} \xi_k\, I\mfm_k(t)=
\sum_{k\geq 1} \xi_k\left(\int_0^t\mfm_k(s)\, ds\right).
\ee

Let  $\mfq=\{q_k,\ k\geq 1\}$ be a sequence of
positive numbers. We write
$$
\mfq^{\ba}:=\prod_{k\geq 1}q_k^{\alpha_k}.
$$

\begin{definition}
\label{def:qGP}
Let $\mfq=\{q_k,\ k\geq 1\}$ be a sequence
such that $0<q_k<1$ for all $k$.

The space $L_{2,\mfq}\big((0,T)\big)$ is  the closure
of $L_2^W\big(\Omega; L_2(0,T)\big)$ with respect to the norm
$$
\|X\|_{2,\mfq}=\left(\sum_{\ba\in \cJ}
\big\|\bE(X\xi_{\ba})\big\|_{L_2((0,T))}^2\right)^{1/2}.
$$
An element of $L_{2,\mfq}\big((0,T)\big)$  is represented by
an expression of the form
$$
X(t)=\sum_{\ba\in \cJ}x_{\ba}(t)\xi_{\ba}
$$
with non-random $x_{\ba}\in L_2((0,T))$ satisfying
$$
\sum_{\ba\in \cJ}  \mfq^{\ba}\,\|x_{\ba}\|_{L_2((0,T))}^2<\infty,
$$
and is called a {\tt  $\mfq$-generalized process}.
\end{definition}

For example, the white noise process
$$
\dot{w}(t)=\sum_{k\geq 1} \xi_k\mfm_k(t)
$$
is a $\mfq$-generalized process for every $\mfq$ satisfying
\bel{q-wn}
\sum_{k\geq 1} q_k<\infty.
\ee

\section{Fractional Derivatives of the Brownian Motion}
\label{sec:FDBM}
\begin{proposition}
If $\gamma\in (0,1)$, then
\bel{BPw-1}
\int_0^t (t-s)^{-\gamma}w(s)\,ds
=\frac{1}{1-\gamma}\int_0^t(t-s)^{1-\gamma}\,dw(s).
\ee
\end{proposition}

\begin{proof}
Integrate by parts on the right-hand side of
\eqref{BPw-1}.
\end{proof}

\begin{corollary}
Given $\gamma\in (0,1)$, define
\bel{Bm-dg}
W_{\gamma}(t)=\frac{1}{(1-\gamma)\Gamma(1-\gamma)}
\int_0^t(t-s)^{1-\gamma}\,dw(s),\ t>0.
\ee
Then
\bel{Jw}
J^{1-\gamma}w(t)=W_{\gamma}(t)=I^{1-\gamma}w(t),
\ee
and
\bel{Bm-dg1}
\partial^{\gamma} w(t)=\frac{d}{dt}W_{\gamma}(t).
\ee
\end{corollary}
\begin{proof}
Equality \eqref{Jw} follows from \eqref{BPw-1}, keeping in mind that $w(0)=0$. After that,  \eqref{C-DJ} implies \eqref{Bm-dg1}.
\end{proof}

\begin{definition}
A process $V=V(t),\ t\in [0,T],$
 is called a {\tt Gaussian Volterra process}
with kernel $K=K(t,s)$ if there exists a non-random
function $K=K(t,s)$ such that
$K(t,s)=0, s>t$, $K\in L_2((0,T)^2)$, and
$$
\mathbb{P}\Big(V(t)=\int_0^t K(t,s)\,dw(s),\ t\in [0,T]\Big)=1.
$$
\end{definition}

\begin{proposition}
\label{prop-pdW}
If
\bel{restr1}
\gamma\in (0,1/2),
\ee
then $\partial^{\gamma}_tw$ is a Gaussian Volterra process with representation
\bel{Pd-bm-reg}
\partial^{\gamma}_tw(t)=\frac{1}{\Gamma(1-\gamma)}
\int_0^t (t-s)^{-\gamma}\, dw(s).
\ee
\end{proposition}

\begin{proof}
Similar to \eqref{lz},
$$
\lim_{t\to 0+}
\int_0^t (t-s)^{-\gamma}w(s)\, ds=0
$$
with probability one. Therefore, it is enough to show that
$$
\int_0^t\left(\int_0^s (s-r)^{-\gamma}\, dw(r)\right)ds=
\frac{1}{1-\gamma}\int_0^t (t-r)^{1-\gamma}\, dw(r),
$$
which follows by the stochastic Fubini theorem; condition \eqref{restr1} is
necessary for the application of the stochastic Fubini theorem.
\end{proof}

Next, for $\beta,\gamma\in (0,1)$, consider the equation
\bel{fSODE1}
\partial^{\beta}_t X(t)=\partial^{\gamma}_tw(t), \ t>0,\  X(0)=X_0,
\ee
Using \eqref{C-DJ}, equation \eqref{fSODE1} becomes
\bel{sol0-1}
\frac{d}{dt}J^{1-\beta}X(t)=\frac{d}{dt}J^{1-\gamma}w(t).
\ee
Together with \eqref{lz}, \eqref{sol0-1} implies that equation
\eqref{fSODE1} should be interpreted as the integral equation
\bel{sol0-2}
 J^{1-\beta}X(t)= J^{1-\gamma}w(t).
 \ee

 \begin{definition}
 \label{def:sol00}
  A classical solution of \eqref{fSODE1}
 on $[0,T]$ is
 a continuous process $X=X(t)$ such that
 $$
 \mathbb{P}\Big( J^{1-\beta}X(t)= J^{1-\gamma}w(t),\
 t\in [0,T]\Big)=1.
 $$
 \end{definition}

\begin{theorem}
\label{th:FD-w}
If
\bel{restr2}
\beta-\gamma>-\frac{1}{2},
\ee
then
\bel{fSODE-sol1}
X(t)=X_0+\frac{1}{\Gamma(1+\beta-\gamma)}
\int_0^t (t-s)^{\beta-\gamma}dw(s)
\ee
is the unique classical solution of \eqref{fSODE1}.
\end{theorem}

\begin{proof}
Apply $\partial^{1-\beta}_t$ to both sides of
 \eqref{sol0-2} and use \eqref{DI=Id}, \eqref{IJ-sg},
 and \eqref{C-DJ}:
 \bel{sol0-3}
 \begin{split}
 X(t)-X_0&=\partial_t^{1-\beta}J^{1-\gamma}w(t)=
 \frac{d}{dt}J^{\beta}J^{1-\gamma}w(t)=
\frac{d}{dt}J^{1+\beta-\gamma}w(t)\\
&=\frac{d}{dt}J^{1-(\gamma-\beta)}w(t);
\end{split}
\ee
note that $1+\beta-\gamma>0$ for all $\beta,\gamma\in (0,1)$.
If $\gamma-\beta>0$, then
\bel{sol0-4}
\frac{d}{dt}J^{1-(\gamma-\beta)}w(t)=\partial_t^{\gamma-\beta}w(t),
\ee
and, under condition \eqref{restr2}, equality \eqref{fSODE-sol1}
follows by Proposition \ref{prop-pdW}.

If $\gamma-\beta\leq 0$, then the function
$t\mapsto J^{1-(\gamma-\beta)}w(t)$ is continuously differentiable in
$t$:
$$
J^{1-(\gamma-\beta)}w(t)=\frac{1}{\Gamma(1+|\gamma-\beta|)}
\int_0^t(t-s)^{|\gamma-\beta|}w(s)ds
$$
 so that
 $$
 \frac{d}{dt}J^{1-(\gamma-\beta)}w(t)=
 \frac{1}{|\gamma-\beta|\Gamma(1+|\gamma-\beta|)}
\int_0^t(t-s)^{|\gamma-\beta|-1}w(s)ds
 $$
 and \eqref{fSODE-sol1} follows after integration by parts.
\end{proof}

We now make the following observations;
 \begin{itemize}
 \item In an ordinary differential equation, $\beta=\gamma=1$,
so that \eqref{restr2} holds.
\item For every $t>0$, $X(t)$ is a Gaussian random variable with
variance
$$
\sigma^2(t)\varpropto \int_0^t  s^{2(\beta-\gamma)}ds
\varpropto t^{2(\beta-\gamma)+1};
$$
 $f\varpropto g$ means $f$ is proportional to $g$.
  The solution of \eqref{fSODE1} can thus exhibit
  the anomalous diffusion behavior $\sigma^2(t)\varpropto
t^{\alpha}$  for all $\alpha\in (0,3)$ \cite[Section 1.2]{FDiff-survey};
 regular diffusion
$\sigma^2(t)\varpropto t$ corresponds to $\beta=\gamma$.
\item The last equality in \eqref{sol0-3} suggests that
the solution of \eqref{fSODE1} can be written as
\bel{fSODE-sol-alt}
X(t)-X(0)=\partial_t^{\gamma-\beta}w(t),
\ee
which makes perfect sense, but requires a justification.
In particular, the proof of Theorem \ref{th:FD-w}
shows that \eqref{restr2} is
necessary for  the right-hand side of \eqref{fSODE-sol-alt} to
define a continuous process, which in this case is a Gaussian
Volterra process.
\item Taking the Laplace transform on both sides of \eqref{fSODE1},
with \eqref{LT-w} in mind, results in
\bel{fd-BM-LT}
\lambda^{\beta}\cL[X](\lambda)-\lambda^{\beta-1}X_0=
\lambda^{\gamma-1}\hat{w}(\lambda)
\ee
or
\bel{fd-BM-LT1}
\cL[X](\lambda)=\lambda^{-1}X_0+\lambda^{\gamma-\beta-1}\hat{w}
(\lambda),
\ee
which is consistent with \eqref{fSODE-sol1}.
\item Condition  \eqref{restr2} is standard in the study of
  fractional stochastic evolution equations \cite{Kim-Fract, LRdS}.
\end{itemize}

If condition \eqref{restr2} fails, so that  $\gamma-\beta\geq 1/2$,
then the solution of \eqref{fSODE1}, as defined by
 \eqref{sol0-2} or \eqref{fd-BM-LT1}, is a generalized process,
best described using weighted chaos spaces.

\begin{theorem}
\label{th:gen0}
If
\bel{FCB}
\mfm_1(t)=\frac{1}{\sqrt{T}},\
\mfm_k(t)=\sqrt{\frac{2}{T}}\cos\left(\frac{\pi t(k-1)}{T}\right),\
k\geq 2,
\ee
and
\bel{no-restr}
\gamma-\beta\geq \frac{1}{2},
\ee
then \eqref{sol0-2} defines a $\mfq$-generalized process
for every $\mfq$ satisfying
 \bel{q-bg}
\sum_k k^{2(\gamma-\beta-1)}q_k<\infty.
\ee
\end{theorem}

\begin{proof}
Using \eqref{w-chaos}, equalities
\eqref{sol0-2} and \eqref{sol0-3} lead to
\bel{chaos000}
X(t)=X_0+\sum_k \xi_k I^{1-(\gamma-\beta)}\mfm_k(t).
\ee
By direct computation,
\bel{cos-asympt}
\left|\int_0^t (t-s)^{-(\gamma-\beta)}\cos(ks)\, ds\right|\leq Ck^{\gamma-\beta-1},
\ee
cf. \cite[Example 6.6.1]{BO-Asympt}, and then \eqref{q-bg} follows.
\end{proof}

Note that
\begin{itemize}
\item As a quick consistency check,
  the extreme case $\beta=0,\gamma=1$ in \eqref{fSODE1}
corresponds to  $X=\dot{w}$,
and then \eqref{q-bg} becomes \eqref{q-wn}.
\item The key step in the proof of Theorem \ref{th:gen0}
is asymptotic analysis, as $ k\to\infty$, of
$$
\int_0^t (t-s)^{-\kappa}\mfm_k(s)ds,\ \kappa\in (0,1),
$$
which is made possible by assumption \eqref{FCB}.
The $k^{-1+\kappa}$-asymptotic holds for other trigonometric
basis, and, while it might not hold in general, the main constructions
related to the chaos expansion, including the definition of the $\mfq$-generalized processes, are intrinsic and do not depend on the
choice of the bases, either in $L_2((0,T))$ or in $L_2^W(\Omega)$;
see \cite{Kond, LR_shir}.
\item Equality  \eqref{chaos000} also holds under \eqref{restr2}.
\end{itemize}

\begin{remark}
With obvious modifications, the results of this section extend to
equations of the type
$$
\partial^{\beta}_tX(t)=\partial^{\gamma}_t\int_0^t g(s)\, dw(s),
$$
where $g\in L_2((0,T))$. Under  condition \eqref{restr2}, the function $g$
can be random as long as $g$ is $\mathcal{F}_t$-adapted and
$$
\bE\int_0^Tg^2(t)\, dt<\infty.
$$
\end{remark}

\section{Time Fractional Ornstein-Uhlenbeck Process}
\label{sec:TFOU}
\subsection{Derivation of the equation}
Consider the harmonic oscillator
\bel{OU1}
m\ddot{x}(t)+c^2x(t)=F(t),
\ee
with a slight twist that the restoring force $-c^2x$ does not
depend on the mass $m$.

The force $F$ has two components, damping $F_d$ and external
$F_e$:
$$
F(t)=F_d(t)+F_e(t).
$$
Traditional damping is
$$
F_d(t)=-c_d\dot{x}(t),\ c_d>0.
$$
Instead, we assume that $F_d$  has memory:
\bel{OU2}
F_d(t)=-\int_0^t f_d(t-s)\dot{x}(s)\, ds.
\ee
A possible choice of the memory kernel in \eqref{OU2} is
$$
f_d(t)=At^{-\beta},\ A>0,\ \beta\in(0,1),
$$
which corresponds to a continuous time
random walk model with a heavy-tailed jump time
distribution; cf. \cite[Section 2.4]{MeerSiko}.

If we also assume that  the external force is the fractional derivative of the
standard Bronwian motion,
$$
F_e=\partial^{\gamma}_tw(t), \ \gamma\in (0,1),
$$
 then, with \eqref{C-d2} in mind, equation \eqref{OU1} becomes
\bel{OU3}
m\ddot{x}(t)+b\partial^{\beta}_tx(t)+c^2x(t)=\partial^{\gamma}_tw(t).
\ee
Now take the Laplace transform on both sides of \eqref{OU3}, with
\eqref{LT-w} in mind.  Assuming zero initial conditions, the result is
\bel{OU4}
\big(m\lambda^2+b\lambda^{\beta}+c^2\big)\, \cL[x](\lambda)=
\lambda^{\gamma-1}\hat{w}(\lambda)
\ee
or
\bel{OU5}
\cL[x](\lambda)=\frac{\lambda^{\gamma-1}\hat{w}(\lambda)}
{m\lambda^2+b\lambda^{\beta}+c^2}.
\ee
Finally, we pass to the limit $m\to 0$ in \eqref{OU5}; this procedure is
known as the Smoluchowski-Kramers
 approximation \cite{Freidlin-SK}. With $X$ denoting the
corresponding limit of $x$, the result is
\bel{OU6}
\cL[X](\lambda)=\frac{\lambda^{\gamma-1}\hat{w}(\lambda)}
{b\lambda^{\beta}+c^2},
\ee
which, back in the time domain, and with re-scaled constants,
becomes the equation describing the
{\em time fractional Ornstein-Uhlenbeck process}:
\bel{OU-main}
\partial^{\beta}_t X(t)=-aX(t)+\partial^{\gamma}_tw(t), \ a>0.
\ee

\subsection{Solution and its long-time behavior}
Similar to Definition \ref{def:sol00}, we say that a continuous process
$X=X(t)$ is a classical solution of \eqref{OU-main} on $[0,T]$ if
$$
\mathbb{P}\left(
J^{1-\beta}X(t)=-a\int_0^tX(s)\,ds + J^{1-\gamma}w(t),\
t\in [0,T]\right)=1.
$$

\begin{theorem}
\label{th:OU-00}
If \eqref{restr2} holds, then, for every $a\in \bR$, $X_0\in \bR$,
 and $T>0$, equation \eqref{OU-main} has a unique solution in the
class $\cE$ and
\bel{OU-sol}
X(t)=X_0E_{\beta}(-at^{\beta})+
\int_0^t
(t-s)^{\beta-\gamma}E_{\beta,\beta-\gamma+1}
\big(-a(t-s)^{\beta}\big)\,dw(s).
\ee
\end{theorem}

\begin{proof}
Take the Laplace transform on both sides of \eqref{OU-main}:
\bel{OU-LT}
\cL[X](\lambda)=\frac{\lambda^{\beta-1}}{\lambda^{\beta}+a}\,X_0+
\frac{\lambda^{\gamma-1}}{\lambda^{\beta}+a}\, \hat{w}(\lambda).
\ee
If $\beta-\gamma>-1/2$, that is, \eqref{restr2} holds,
then inverting  \eqref{OU-LT} yields \eqref{OU-sol}.
\end{proof}

 Equality \eqref{OU-sol} implies that,  for every $t>0$,
 $X(t)$ is a Gaussian random variable with mean
$$
\mu(t)=X_0E_{\beta}(-at^{\beta})
$$
 and variance
\bel{OU-var}
\sigma^2(t)=\int_0^{t}
s^{2(\beta-\gamma)}E^2_{\beta,\beta-\gamma+1}(-as^{\beta})\,ds.
\ee

By  \eqref{ML-asympt1},
$$
\lim_{t\to +\infty} \mu(t)=0.
$$
To study $\sigma^2(t)$, we use
\eqref{ML-asympt1-1} with $\rho=\beta-\gamma+1$:
\bel{ML-asympt2}
\lim_{t\to+\infty} tE_{\beta,\beta-\gamma+1}(-t)=
\frac{1}{\Gamma(1-\gamma)}.
\ee
Note that \eqref{restr2} is necessary and sufficient for the convergence of
\eqref{OU-var}  at zero. On the other hand, \eqref{ML-asympt2}
implies that, depending on the values of $\gamma$,
 the integral in \eqref{OU-var}
 can either converge or diverge at infinity:
\bel{OU-asypt0}
 s^{2(\beta-\gamma)}E^2_{\beta,\beta-\gamma+1}(-as^{\beta})
 \sim s^{-2\gamma},\ s\to +\infty.
 \ee

 Using \eqref{OU-asypt0}, as well as the arguments similar to
 the proof of Theorem \ref{th:gen0}, we get the following
  characterization of the solution of
 \eqref{OU-main} for all values of $\beta,\gamma\in (0,1]$.
 \begin{theorem}
 \begin{enumerate}
 \item If $\beta-\gamma\leq -1/2$, then $X$ is a $\mfq$-generalized
 process for every $\mfq$ satisfying \eqref{q-bg};
 \item If $\beta-\gamma> -1/2$, then $X$ is a Gaussian Volterra process \eqref{OU-sol} and
     \begin{itemize}
     \item If $\gamma>1/2$, then, as $t\to +\infty$, $X(t)$ converges in distribution to a Gaussian random variable with mean zero
         and variance
\bel{FOU-LimVar}
 \sigma^2_{\infty}(a,\beta,\gamma)=\int_0^{+\infty}
s^{2(\beta-\gamma)}E^2_{\beta,\beta-\gamma+1}(-as^{\beta})\,ds;
\ee
\item If $\gamma=1/2$, then, as $t\to \infty$, $X(t)$ is
a Gaussian random variable with mean of order $t^{-\beta}$
and variance of order
$\ln t$;
\item If $\gamma<1/2$, then, as $t\to \infty$, $X(t)$ is
a Gaussian random variable with mean of order $t^{-\beta}$
 and variance of order $ t^{1-2\gamma}$.
     \end{itemize}
 \end{enumerate}
 \end{theorem}
In particular, for $\beta\in(0,1] $ and $\gamma\in (0,1/2)$, the
long-time behavior of $X$ corresponds to that of a
{sub-diffusion}; see, for example,  \cite{Sokolov-1}
or \cite[Section 6]{MK-chem}.
 For $\gamma=1/2$, the result is an ultra-slow,
or Sinai-type, diffusion \cite{Sinai-diffusion}.

\begin{remark}
Different physical considerations  lead to alternative forms of
the time-fractional Ornstein-Uhlenbeck process: see, for example
\cite{Kwok} and references therein.
\end{remark}

\section{Time Fractional Geometric Brownian Motion}
\label{sec:TFGBM}

Similar to the geometric Brownian motion
$$
dx(t)=ax(t)dt+\sigma x(t)dw(t),
$$
which is
\bel{GBM}
x(t)=x(0)\exp\left(\left(a-\frac{\sigma^2}{2}\right)t+\sigma w(t)\right),
\ee
define the time  fractional geometric Brownian motion as the
solution of the equation
\bel{FGBM}
\partial^{\beta}_tX(t)=aX(t)+\sigma
\partial_t^{\gamma}\int_0^tX(s)\,dw(s),
\ t>0,\ \ \beta,\gamma\in (0,1),
\ee
with non-random initial condition $X(0)=X_0$.

By Theorem \ref{th:FD-w}, if  $\gamma-\beta<1/2$, then \eqref{FGBM}
 is equivalent to the integral equation
\bel{FGBM-I}
X(t)=X_0+
\frac{a}{\Gamma(\beta)}\int_0^t (t-s)^{\beta-1}X(s)\,ds+
\frac{\sigma}{\Gamma(1+\beta-\gamma)}
\int_0^t (t-s)^{\beta-\gamma}X(s)\,dw(s);
\ee
using \eqref{OU-sol}, we get  a different, but equivalent,  equation
\bel{FGBM-I-OU}
X(t)=X_0E_{\beta}(at^{\beta})+
\sigma \int_0^t
(t-s)^{\beta-\gamma}E_{\beta,\beta-\gamma+1}
\big(a(t-s)^{\beta}\big)X(s)\,dw(s).
\ee
Accordingly, we define a classical solution of \eqref{FGBM}
on $[0,T]$  as a continuous process $X=X(t)$
such that, for all $t\in [0,T]$,  $\bE X^2(t)<\infty$,
$X(t)$ is $\mathcal{F}_t$-measurable,  and
\eqref{FGBM-I} holds with probability one.

Because  a closed-form expression of the type \eqref{GBM} is currently
not available for $X(t)$,  we will study  \eqref{FGBM} using chaos
expansion.

To simplify the notations, let
\bel{FSol}
\Phi(t)=t^{\beta-\gamma}E_{\beta,\beta-\gamma+1}(at^{\beta}).
\ee

\begin{theorem}
\label{th:FGBM}
Under condition \eqref{restr2}, equation \eqref{FGBM}
has a unique classical solution for every $T>0$ and $X_0,a,\sigma\in \bR$, with chaos expansion
$$
X(t)=\sum_{\ba\in \cJ} X_{\ba}(t) \xi_{\ba},
$$
where
\bel{Xal-cl}
\begin{split}
\sum_{|\ba|=n}&X_{\ba}(t)\xi_{\ba}=X_0\sigma^n
\int_0^t\int_0^{s_n}\ldots\int_0^{s_2}\\
&\Phi(t-s_n)\Phi(s_n-s_{n-1})\ldots\Phi(s_2-s_1)
E_{\beta}(as_1^{\beta})\,dw(s_1)\ldots dw(s_n),
\end{split}
\ee
and
\bel{Xal-sq}
\begin{split}
\bE X^2(t) &= X_0^2\Big(
E_{\beta}^2(at^{\beta})+\sum_{n=1}^{\infty}\sigma^{2n}
\int_0^t\int_0^{s_n}\ldots\int_0^{s_2}\\
&\Phi^2(t-s_n)\Phi^2(s_n-s_{n-1})\ldots\Phi^2(s_2-s_1)
E_{\beta}^2(as_1^{\beta})\,ds_1\ldots ds_n\Big).
\end{split}
\ee
\end{theorem}

\begin{proof}
Existence and uniqueness follow from \eqref{FGBM-I} by the
standard fixed point argument. To derive \eqref{Xal-cl},
we use the general result about chaos expansion for linear
evolution equations \cite[Section 6]{LR_shir}.
In particular, the functions $X_{\ba}=X_{\ba}(t)$, $\ba\in \cJ$,
satisfy a system of equations, known as the {\tt propagator}:
\bel{prop00}
\begin{split}
|\ba|=0: \
\partial^{\beta}_t X_{\zm}&=aX_{\zm},\ X_{\zm}(0)=X_0;\\
|\ba|>0:\
\partial^{\beta}_tX_{\ba}(t)&=aX_{\ba}(t)
+\sigma\sum_{k\geq 1} \sqrt{\alpha_k}\,
I^{1-\gamma}(X_{\ba-\km}\,\mfm_k)(t), \ X_{\ba}(0)=0,
\end{split}
\ee
where $\km$ is the multi-index with $|\km|=1$ and the only
non-zero element in position $k$. By Theorem \ref{th:DetEq-Gen},
\bel{prop00-sol}
\begin{split}
X_{\zm}(t)&=X_0E_{\beta}(at^{\beta}),\\
X_{\ba}(t)&=
\sigma\sum_{k\geq 1} \sqrt{\alpha_k}
\int_0^t(t-s)^{\beta-1}E_{\beta,\beta}\big(a(t-s)^{\beta}\big)
I^{1-\gamma}(X_{\ba-\km}\,\mfm_k)(s)\,ds,\\
 & |\ba|>0.
\end{split}
\ee
Changing the order of integration and using \eqref{MTL-con-gen},
\begin{align*}
&\int_0^t(t-s)^{\beta-1}E_{\beta,\beta}\big(a(t-s)^{\beta}\big)
I^{1-\gamma}(X_{\ba-\km}\,\mfm_k)(s)\, ds\\
=
&\int_0^t(t-s)^{\beta-\gamma}E_{\beta,1+\beta-\gamma}
\big(a(t-s)^{\beta}\big)
X_{\ba-\km}(s)\mfm_k(s)\, ds,
\end{align*}
and then, iterating the result,
\bel{Xba-sol}
\begin{split}
X_{\ba}(t)&=\frac{\sigma^n}{\sqrt{\ba!}}
\sum_{\pi\in \mathcal{P}^n}
\int_0^t\int_0^{s_n}\cdots\int_0^{s_2}
\Phi(t-s_n)\\
&\times\Phi(s_n-s_{n-1})\ldots\Phi(s_2-s_1) X_{\zm}(s_1)
\mfm_{i_{\pi(n)}}(s_n)\ldots\mfm_{i_{\pi(1)}}\,
ds_1\ldots ds_n,
\end{split}
\ee
where $\mathcal{P}^n$ is the permutations group of
$\{1,\ldots,n\}$
and $\{i_1,\ldots,i_n\}$ is the
characteristic set of $\ba$; cf. \cite[Corollary 6.6]{LR_shir}.
After that, \eqref{Xal-cl} follows from the connection between
the Hermite polynomials and the iterated It\^{o} integrals
\cite[Theorem 3.1]{Ito}. Then \eqref{Xal-sq} follows from
\eqref{Xal-cl} by It\^{o} isometry.

It remains to show that the right-hand side of \eqref{Xal-sq} is
finite. To this end, we use \eqref{ML-asympt0} to write
$$
\Phi^2(t)\leq Ct^{r-1},\ r=2(\beta-\gamma)+1>0,
$$
so that
\begin{equation*}
\begin{split}
\int_0^t\int_0^{s_n}&\ldots\int_0^{s_2}
\Phi^2(t-s_n)\Phi^2(s_n-s_{n-1})\ldots\Phi^2(s_2-s_1)\\
&\times E_{\beta}^2(as_1^{\beta})\,ds_1\ldots ds_n\Big)
\leq \frac{C^n(T)}{\Gamma(nr+1)},
\end{split}
\end{equation*}
and convergence follows by the Stirling formula.
\end{proof}

Note that condition \eqref{restr2} is necessary and sufficient for the
convergence of the integrals in \eqref{Xal-sq}, and once again we
see that, without \eqref{restr2}, no classical solution of \eqref{FGBM}
can exist.

On the other hand, each $X_{\ba}$ is well-defined
by \eqref{prop00-sol} for all $\beta,\gamma\in (0,1]$.
Accordingly, we call the resulting formal sum  $\sum_{\ba\in \cJ}
X_{\ba}(t)\xi_{\ba}$ the {\tt chaos solution} of
\eqref{FGBM}. By construction, this solution exists and is unique.

\begin{theorem}
If $\beta-\gamma\leq -1/2$ and $\{\mfm_k,\ k\geq 1\}$
 are given by \eqref{FCB}, then the chaos solution of
\eqref{FGBM} is a $\mfq$-generalized process for
every $\mfq$ satisfying \eqref{q-bg}.
\end{theorem}

\begin{proof}
The objective is to show that \eqref{q-bg} implies
\bel{Xal-gen-sq}
\sum_{\ba\in \cJ}\mfq^{\ba}|X_{\ba}(t)|^2<\infty, \ t>0,
\ee
and analysis of the proof of Theorem \ref{th:FGBM} shows that
\eqref{Xal-gen-sq} will follow from
\bel{Phi-gen-sq}
\sum_{k\geq 1} q_k\left(\int_0^T \Phi(T-s)\,\mfm_k(s)\, ds\right)^2
<\infty.
\ee
To prove \eqref{Phi-gen-sq}, define the operator
$Q$ on $L_2((0,T))$ by
$$
Q\mfm_k=\sqrt{q_k}\,\mfm_k,\ k\geq 1.
$$
Then the operator  $Q$ is symmetric on $L_2((0,T))$,
\begin{align*}
Qf(t)&=\sum_{k\geq 1} \sqrt{q_k}
\left(\int_0^Tf(s)\mfm_k(s)\, ds\right)
\mfm_k(t),\\
\sum_{k\geq 1} & q_k\left(\int_0^T \Phi(T-s)\,\mfm_k(s)\, ds\right)^2
=\int_0^T\big(Q\Phi(t)\big)^2\,dt,
\end{align*}
and \eqref{Phi-gen-sq} follows from \eqref{cos-asympt}.
\end{proof}

\section{Stochastic Fractional Parabolicity Conditions}
\label{sec:SFPC}

Consider the stochastic equation
\bel{StochParab}
du(t,x)=bu_{xx}(t,x)dt +
\big(\varrho u_{xx}(t,x)+
\sigma u_x(t,x)+cu(t,x)\big)\, dw(t),\ t>0,\ x\in \bR,
\ee
with real numbers $b,\varrho, \sigma, c$ as parameters.
It is well known that
\begin{itemize}
\item Equation \eqref{StochParab} is well-posed in $L_2(\bR)$ if and
only if $\varrho=0$ and $2b-\sigma^2\geq 0$; see, for example,
\cite[Section 2.3.1]{LR-SPDE}.
\item Equation \eqref{StochParab} is well-posed in  a suitable chaos
space if $b>0$; cf. \cite{LR-fs}.
\end{itemize}
On other hand, a perturbation-type argument  \cite{Kim-Fract}
shows that the following fractional version of
\eqref{StochParab},
\bel{F-SPDE0}
\partial_t^{\beta}u(t,x)=
au_{xx}(t,x)+\partial_t^{\gamma}\int_0^t
\big(\varrho u_{xx}(s,x)+\sigma u_x(s,x) + c u(s,x)\big)\, dw(s)
\ee
is well-posed in $L_2(\bR)$ if $\beta\in (0,1)$,
$|\varrho|$ is sufficiently close to zero, and $0<\gamma<1/2$.  Note that if
$\gamma<1/2$, then \eqref{restr2} holds for all $\beta\in (0,1)$.

The objective of this section is to  establish more general sufficient conditions for
 well-posedness of \eqref{F-SPDE0} and similar equations.

Fix the numbers  $b>0,\ \sigma\in \bR,\
\beta,\gamma\in (0,1]$, $\alpha,\nu\in (0,2]$, and let
$$
\Lambda=(-\boldsymbol{\Delta})^{1/2}
$$
be the fractional Laplacian defined in the Fourier domain by
$$
\frac{1}{(2\pi)^{\mathrm{d}/2}}
\int_{\bRd}e^{-i x y} (\Lambda f)(x)\, dx =
 \frac{|y|}{(2\pi)^{\mathrm{d}/2}}\int_{\bRd}e^{-i x y}  f(x)\, dx.
$$

Consider the  equation
\bel{F-SPDE1}
\partial_t^{\beta}u(t,x)+
b\Lambda^{\alpha}(t,x)=\sigma \,\partial_t^{\gamma}\!\!\int_0^t
 \Lambda^{\nu}u(s,x)\, dw(s),\ t>0,\ x\in \bRd,
\ee
with non-random initial condition $u(0,\cdot)\in L_2(\bRd)$.

\begin{definition}
An $\mathcal{F}_t$-adapted process
$u\in L_2\Big(\Omega; \cC\big([0,T];
L_2(\bRd)\big)\Big)$ is called a solution of \eqref{F-SPDE1} if,
for every  $\varphi\in \cC^{\infty}_0(\bRd),$
\begin{align*}
\mathbb{P}\Bigg(&
J^{1-\beta}\big(u,\varphi\big)_{L_2(\bRd)}(t)+
b\int_0^t\big(u,\Lambda^{\alpha}\varphi\big)_{L_2(\bRd)}(s)\,ds\\
&=\sigma J^{1-\gamma}\left(\int_0^{\cdot}\big(u,
\Lambda^{\nu}\varphi\big)_{L_2(\bRd)}(s)\,dw(s)\right)(t),\
 t\in [0,T]\Bigg)=1.
\end{align*}
Equation  \eqref{F-SPDE1} is called well-posed in $L_2(\bRd)$ if, for
every initial condition $u(0,\cdot)\in \cC^{\infty}_0(\bRd)$, there
exists a unique solution $u$ and
$$
\bE\|u\|_{L_2(\bRd)}^2(t)\leq C \|u(0,\cdot)\|_{L_2(\bRd)}^2,
$$
with $C$ independent of the initial condition.
\end{definition}

The $L_2$-isometry of the Fourier transform implies that,
  in terms of  well-posedness in $L_2$, equation
\eqref{F-SPDE1} with $\alpha=2$ and $\nu=1,2$, is equivalent to \eqref{F-SPDE0}; this equivalence
might no longer hold  for well-posedness in $L_p$, $p>2$,
 see \cite{BV-Lp} when $\beta=\gamma=1,\ \alpha=2, \ \nu=1$.

 In the rest of the section we show that, under \eqref{restr2},
 \begin{enumerate}
 \item For  $\gamma\in (0,1/2)$,
 equation \eqref{F-SPDE1} is well-posed in $L_2(\bRd)$
 if and only if $\alpha\geq \nu$;
 \item For  $\gamma=1/2$,
 equation \eqref{F-SPDE1} is well-posed in $L_2(\bRd)$
 if and only if $\alpha> \nu$;
\item For $\gamma=(1/2)+\beta\varepsilon,$
$\varepsilon\in (0,1)$, equation \eqref{F-SPDE1} is well-posed in $L_2(\bRd)$ if  $\alpha> \nu/(1-\varepsilon)$, and
can be well-posed when $\alpha= \nu/(1-\varepsilon)$
under additional conditions on $b$ and $\sigma$.
\end{enumerate}

\begin{theorem}
\label{th-SFPC}
Assume  that  \eqref{restr2} holds.
 Then \eqref{F-SPDE1} is
well-posed in $L_2(\bRd)$ in each of the following cases:
\begin{itemize}
\item $\gamma\in (0,1/2)$ and $\alpha\geq \nu$;
\item $\gamma=1/2$ and $\alpha>\nu$;
\item $\gamma\in (1/2,1]$ and
\bel{restr1s}
\alpha>\frac{\nu}{1-\frac{\gamma-(1/2)}{\beta}}.
\ee
\end{itemize}
\end{theorem}

\begin{proof}
Denote by $U=U(t,y)$  the Fourier transform of $u$ in the space variable.
Then, by Fourier isometry, \eqref{F-SPDE1} is well-posed in
$L_2(\bRd)$ if and only if
\bel{F-WP}
\bE|U(t,y)|^2\leq C|U(0,y)|^2,\ t>0,
\ee
for some $C$ independent of $y$.
Accordingly, throughout the proof,   $C$ denotes a positive number independent of $y$.

Equation  \eqref{F-SPDE1} in Fourier domain is
\bel{F-SPDE-FD}
\partial^{\beta}_t U(t,y)=
-b|y|^{\alpha}\,U(t,y)+\sigma|y|^{\nu}\,
\partial^{\gamma}_t\!\!\int_0^t U(s,y)\, dw(s).
\ee
Notice that, for each $y\in \bRd$, equation
 \eqref{F-SPDE-FD} is of the same type as
\eqref{FGBM}. Then \eqref{FGBM-I-OU} implies
\bel{F-SPDE-FD-sq}
\begin{split}
\bE |U(t,y)|^2&=
|U(0,y)|^2E^2_{\beta}\big(-b|y|^{\alpha}t^{\beta}\big)\\
&+
\sigma^2|y|^{2\nu}\int_0^t
(t-s)^{2(\beta-\gamma)} E^2_{\beta,\beta-\gamma+1}
\big(-b|y|^{\alpha}(t-s)^{\beta}\big)\,\bE|U(s,y)|^2\,ds.
\end{split}
\ee

If $|y|\leq 1$, then \eqref{ML-asympt1} immediately implies
$$
\bE |U(t,y)|^2\leq C\Big( |U(0,y)|^2
+\int_0^t(t-s)^{2(\beta-\gamma)}\,\bE |U(s,y)|^2\, ds\Big),
$$
and \eqref{F-WP} follows by Proposition \ref{prop-FGBI}.

If  $|y|>1$, and $\gamma\in (0,1/2)$,
 then we use \eqref{ML-asympt1} to write
 \begin{align}
 E^2_{\beta}\big(-b|y|^{\alpha}t^{\beta}\big)& \leq C,
 \notag \\
 \label{MainEst}
E^2_{\beta,\beta-\gamma+1}
\big(-b|y|^{\alpha}(t-s)^{\beta}\big)&\leq
\frac{C}{b^2|y|^{2\alpha}(t-s)^{2\beta}},
\end{align}
and then \eqref{F-SPDE-FD-sq} becomes
\bel{F-SPDE-FD-sd}
\bE |U(t,y)|^2\leq C\Big(|U(0,y)|^2+
\int_0^t(t-s)^{-2\gamma}\,\bE|U(s,y)|^2\,ds\Big),
\ee
so that \eqref{F-WP} again follows by Proposition \ref{prop-FGBI}.

If $\gamma\geq1/2$, then the integral on the right-hand side of
\eqref{F-SPDE-FD-sd} diverges. Accordingly, we replace \eqref{MainEst}
with
$$
E^2_{\beta,\beta-\gamma+1}
\big(-b|y|^{\alpha}(t-s)^{\beta}\big)\leq
\frac{C}{\big(a^2|y|^{2\alpha}(t-s)^{2\beta}\big)^{1-\varepsilon}},
$$
 taking $\varepsilon>0$ if $\gamma=1/2$ and
 $\varepsilon\beta>\gamma-1/2$ if $\gamma\in (1/2,1]$.
 Note that \eqref{restr2} is equivalent to  $1-\varepsilon>0$.
 Then, instead of \eqref{F-SPDE-FD-sd},  we get
 $$
 \bE |U(t,y)|^2\leq C\Big(|U(0,y)|^2+
\int_0^t(t-s)^{2(\varepsilon\beta-\gamma)}\,\bE|U(s,y)|^2\,ds\Big),
$$
as long as $\alpha(1-\varepsilon)\geq \nu$, and conclude the
 proof by applying Proposition \ref{prop-FGBI}.
\end{proof}

Note that
\begin{enumerate}
\item If $\gamma\geq (\beta+1)/2$, then \eqref{restr1s}
becomes $\alpha>2\nu$. For equation \eqref{F-SPDE0},
this means $\varrho=\sigma=0$, which is consistent with
\cite{Kim-Fract}.
\item The results of \cite{LR-fs} suggest that \eqref{F-SPDE1}
is unlikely to have a $\mfq$-generalized chaos
solution when $\nu>\alpha$.
\end{enumerate}

 If $\gamma\in (0,1/2)$, then condition $\alpha\geq \nu$
is also necessary: \eqref{ML-asympt1-1} shows that
\eqref{F-WP} is not possible for large $|y|$ when $\alpha<\nu$.
If $\gamma=\beta=1$, then \eqref{restr1s}
becomes $\alpha>2\nu$. On the other hand, similar to
\eqref{StochParab}, equation
\eqref{F-SPDE1} is well-posed in $L_2(\bRd)$ if $\gamma=\beta=1$,
$\alpha=2, \ \nu=1$, and
\bel{WP-main}
2b\geq \sigma^2.
\ee
This observation suggests that, more generally,   \eqref{F-SPDE1}
could be well-posed if $\gamma- 1/2=\beta\varepsilon$,
$\varepsilon\in (0,1),$ and $\alpha=\nu/(1-\varepsilon)$,
under an additional condition of the type \eqref{WP-main}.
Dimensional analysis implies that the  condition should be of the form
$b\geq C(\beta,\gamma)|\sigma|^{1/(1-\varepsilon)}$.
We conclude this section by  establishing an upper bound for $C(\beta,\gamma)$,
 as well as addressing a similar question when $\gamma=1/2$.

 \begin{theorem}
\label{th-SFPC-eq}
Assume  that  \eqref{restr2} holds.
 Then equation \eqref{F-SPDE1} is
 \begin{itemize}
\item NOT well-posed in $L_2(\bRd)$ if
$\gamma=1/2$ and $\alpha=\nu$;
\item well-posed in $L_2(\bRd)$ if
$\gamma=(1/2)+\varepsilon\beta$, $\varepsilon\in (0,1)$,
\bel{restr1s-eq}
\alpha=\frac{\nu}{1-\varepsilon},\
b\geq
 \big(\sigma_{\infty}^2(1,\beta,\gamma)\big)^{1/(2-2\varepsilon)}
 \,|\sigma|^{1/(1-\varepsilon)},
\ee
\end{itemize}
with $\sigma_{\infty}^2(a,\beta,\gamma)$  defined in
 \eqref{FOU-LimVar}.
\end{theorem}

\begin{proof}
Similar to the proof of the previous theorem, we need to study
equality \eqref{F-SPDE-FD-sq} for $|y|>1$, so we fix $y$
with $|y|$ sufficiently large  and define
$$ V(t)=\bE|U(t,y)|^2-|U(0,y)|^2E^2_{\beta}(-b|y|^{\alpha}t^{\beta}).
$$
Then  $V$ is non-decreasing in $t$ and satisfies
\bel{F-SPDE-FD-sq-m}
\begin{split}
V(t)&=
\sigma^2|y|^{2\nu}\int_0^t
(t-s)^{2(\beta-\gamma)} E^2_{\beta,\beta-\gamma+1}
\big(-b|y|^{\alpha}(t-s)^{\beta}\big)V(s)\,ds\\
&+|U(0,y)|^2\
\sigma^2|y|^{2\nu}\int_0^t
s^{2(\beta-\gamma)} E^2_{\beta,\beta-\gamma+1}
\big(-b|y|^{\alpha}s^{\beta}\big)
E^2_{\beta}\big(-b|y|^{\alpha}(t-s)^{\beta}\big)\,ds.
\end{split}
\ee

If $\gamma=1/2$, we re-write \eqref{F-SPDE-FD-sq-m} as
$$
V(t)\geq
|U(0,y)|^2
\sigma^2|y|^{2\nu}\int_0^t
s^{2\beta-1} E^2_{\beta,\beta-\gamma+1}
\big(-b|y|^{\alpha}s^{\beta}\big)
E^2_{\beta}\big(-b|y|^{\alpha}(t-s)^{\beta}\big)\,ds.
$$
Changing the variables
$T=(b|y|^{\alpha})^{1/\beta}\, t$,
$\tau=(b|y|^{\alpha})^{1/\beta}s$,
and keeping in mind that $\nu=\alpha$,
$$
V(t)\geq
|U(0,y)|^2
\sigma^2b^{-2}\int_0^T
\tau^{2\beta-1} E^2_{\beta,\beta-\gamma+1}
\big(-\tau^{\beta}\big)
E^2_{\beta}\big( (T-\tau)^{\beta}\big)\,d\tau.
$$
Because $\lim_{|y|\to \infty} T=+\infty$ and, by
\eqref{ML-asympt1-1}, the last integral diverges at infinity,
we conclude that \eqref{F-WP} cannot hold.

Next, consider the case  $\gamma\in (1/2,1]$ under the
assumptions  \eqref{restr1s-eq}. The same computations as
in the case $\gamma=1/2$ show that now the
second integral on the right-hand side of  \eqref{F-SPDE-FD-sq-m}
is uniformly bonded in $|y|$. To analyze the first integral,
write
$$
\varepsilon=\frac{\gamma-(1/2)}{\beta},\
\varpi=|y|^{-\varepsilon \alpha/\beta},\
$$
so that $\nu=\alpha(1-\varepsilon)$, and
\bel{AUX1}
\begin{split}
&\sigma^2|y|^{2\nu}\int_0^t
(t-s)^{2(\beta-\gamma)} E^2_{\beta,\beta-\gamma+1}
\big(-b|y|^{\alpha}(t-s)^{\beta}\big)V(s)\,ds\\
&=\sigma^2|y|^{2\nu}
\left(\int_0^{t-\varpi}+\int_{t-\varpi}^t\right)
(t-s)^{2(\beta-\gamma)} E^2_{\beta,\beta-\gamma+1}
\big(-b|y|^{\alpha}(t-s)^{\beta}\big)V(s)\,ds.
\end{split}
\ee
We use \eqref{MainEst} and $t-s\geq \varpi$
to bound the first integral on the
right-hand side of \eqref{AUX1} by
$$
\frac{C\sigma^2 |y|^{2\nu}}{\big(b|y|^{\alpha}\varpi^{\beta}\big)^2}
\int_0^{t-\varpi}(t-s)^{2(\beta-\gamma)}V(s)\, ds
\leq C \int_0^t (t-s)^{2(\beta-\gamma)}V(s)\, ds,
$$
which, by \eqref{restr2}, allows an application of
Proposition \ref{prop-FGBI}.
For the second integral, we use monotonicity of $V$ to get an
upper bound
\begin{align*}
V(t)\, &\sigma^2|y|^{2\nu} \int_{t-\varpi}^t
(t-s)^{2(\beta-\gamma)} E^2_{\beta,\beta-\gamma+1}
\big(-b|y|^{\alpha}(t-s)^{\beta}\big) \,ds\\
&=
V(t)\, \sigma^2|y|^{2\nu} \int_0^{\varpi}
s^{2(\beta-\gamma)} E^2_{\beta,\beta-\gamma+1}
\big(-b|y|^{\alpha}s^{\beta}\big) \,ds,
\end{align*}
which, after the change of variable $\tau=(b|y|^{\alpha})^{1/\beta}s$
becomes
$$
V(t)\, \frac{\sigma^2}{b^{2(1-\varepsilon)}}
\int_0^{b^{1/\beta}|y|^{\nu/\beta}}
\tau^{2(\beta-\gamma)} E^2_{\beta,\beta-\gamma+1}
\big(-\tau^{\beta}\big) \,d\tau
<V(t)\,\frac{\sigma^2}{b^{2(1-\varepsilon)}}\,
\sigma_{\infty}^2(1,\beta,\gamma).
$$
By assumption,
$$
\frac{\sigma^2}{b^{2(1-\varepsilon)}}\,
\sigma_{\infty}^2(1,\beta,\gamma)\leq 1,
$$
and then \eqref{F-WP} follows from \eqref{F-SPDE-FD-sq-m}.
\end{proof}

As a final comment, note that, while the proof of Theorem
\ref{th-SFPC-eq} suggests that \eqref{restr1s-eq} might not
be sharp, we do get the optimal bound \eqref{WP-main}
 when $\alpha=2$, $\nu=1$,
$\beta=\gamma=1$, and $\varepsilon=1/2$, because,
with $E_{1,1}(t)=e^t$,
$$
\sigma_{\infty}^2(1,1,1)=\int_0^{\infty} e^{-2t}\, dt=\frac{1}{2}.
$$

\def\cprime{$'$}
\providecommand{\bysame}{\leavevmode\hbox to3em{\hrulefill}\thinspace}
\providecommand{\MR}{\relax\ifhmode\unskip\space\fi MR }
\providecommand{\MRhref}[2]{%
  \href{http://www.ams.org/mathscinet-getitem?mr=#1}{#2}
}
\providecommand{\href}[2]{#2}


\end{document}